\documentclass[twocolumn,final]{svjour3}

\usepackage{times}
\usepackage{amsthm}
\usepackage{amsfonts}
\usepackage{mathtools}
\usepackage{amsmath}
\usepackage{amssymb}
\usepackage{xcolor}
\usepackage{bm}
\usepackage{stanli}
\usepackage{pgfplots}
\usepackage{subcaption}
\usepackage[round]{natbib}
\usepackage{hyperref}

\title{Global weight optimization of frame structures with polynomial programming}
\author{Marek Tyburec \and Michal Ko{\v{c}}vara \and Martin Kru\v{z}\'ik}

\newtheorem{assumption}{Assumption}

\newcommand*\circled[1]{\tikz[baseline=(char.base)]{
		\node[shape=circle,draw,inner sep=0pt,fill=white, minimum size=4mm] (char) {#1};}}
\newcommand*\squared[1]{\tikz[baseline=(char.base)]{
		\node[shape=rectangle,draw,inner sep=0pt, minimum size=4mm] (char) {#1};}}

\institute{Marek Tyburec\textsuperscript{1,2}\\
	\email{marek.tyburec@cvut.cz}\\
	ORCID: 0000-0003-0798-0948\\
	\\
	Michal Ko{\v{c}}vara\textsuperscript{2,3}\\
	\email{m.kocvara@bham.ac.uk}\\
	ORCID: 0000-0003-4414-0083\\
	\\
	Martin Kru{\v{z}}\'ik\textsuperscript{1,2}\\
	\email{kruzik@utia.cas.cz}\\
	ORCID: 0000-0003-1558-5809\\
	\\\
	\textsuperscript{1}Faculty of Civil Engineering, Czech Technical University in Prague, Th\'{a}kurova 7, 16629 Prague 6, Czech Republic\\
	\\
	\textsuperscript{2}Institute of Information Theory and Automation, Czech Academy of Sciences, Pod Vod{\'a}renskou v{\v{e}}{\v{z}}{\'i} 1143/4, 18200 Prague 8, Czech Republic\\
	\\
	\textsuperscript{3}School of Mathematics, The University of Birmingham, 3 Birmingham B15 2TT, United Kingdom
}

\begin{document}

\maketitle

\begin{abstract}
Weight optimization of frame structures with continuous cross-section parametrization is a challenging non-convex problem that has traditionally been solved by local optimization techniques. Here, we exploit its inherent semi-algebraic structure and adopt the Lasserre hierarchy of relaxations to compute the global minimizers. While this hierarchy generates a natural sequence of lower bounds, we show, under mild assumptions, how to project the relaxed solutions onto the feasible set of the original problem and thus construct feasible upper bounds. Based on these bounds, we develop a simple sufficient condition of global $\varepsilon$-optimality. Finally, we prove that the optimality gap converges to zero in the limit if the set of global minimizers is convex. We demonstrate these results by means of two academic illustrations.
\end{abstract}


\section{Introduction}

Finding cross-section parameters that minimize the weight of frame structures for given performance constraints constitutes a fundamental problem of structural design \citep{Bendse2004,Saka_2013}. This problem naturally arises, e.g., in civil~\citep{Thevendran_1992,Mosquera_2014}, automotive \citep{Zuo_2016}, or machine \citep{Tyburec2019} industries.

In contrast to optimization of trusses, for which several convex formulations were established thanks to the linear dependence of the stiffness matrix on the cross-section areas \citep{Stolpe_2017,Kocvara_2017,Bendse2004}, the stiffness of frame elements is non-linear due to the non-linear coupling of cross-section areas with moments of inertia. Therefore, the emergent optimization problems are non-convex in general and very challenging to be solved globally \citep{Yamada_2015,Tyburec2021,Toragay2022}. 

Due to the non-convexity, majority of methods for optimizing frame structures are either local, or (meta-)heuristic \citep{Saka_2013}, thus converging to solutions of unknown quality with respect to the global minimizers. For example, \citet{Saka_1980} optimized the weight of frame structures while accounting for stress and displacement constraints by using sequential linear programming, and \citet{Wang_2006} improved over its solution efficiency by adopting sequential quadratic programming instead. For the same setting, \citet{Khan_1984} and \citet{Chan_1995} developed optimality criteria methods. Further, \citet{Yamada_2015} optimized the weight of frame structures for a prescribed fundamental free-vibration eigenfrequency lower bound by using a sequence of semidefinite programming relaxations, concluding that good-quality local optima are attained.

Several optimization methods have also been developed for the discrete setting of the frame optimization problem, i.e., considering a catalog of available cross-sections. The associated optimization methods are naturally enumerative, with the global optimizers reachable by branch-and-bound-type methods. For example, \citet{Kureta_2013,Hirota_2015} and \citet{Van_Mellaert_2017} formulated mixed-integer linear programs, and \citet{Kanno_2016} developed a mixed-integer second-order conic programming formulation. Another related approach was introduced by \citet{Wang2021}, who solved a single semidefinite programming relaxation whose solution served as an input to a (meta-heuristic) neighborhood search of a differential evolution algorithm. Many other heuristic and meta-heuristic approaches were presented, we refer the reader to \citep{Saka_2013} for an extensive review.
 
Returning to the continuous case, which is the purpose of this work, only three global approaches have been presented to the best of our knowledge. First, \citet{Toragay2022} tackled the displacement-constrained weight minimization problem of frame structures by casting it into a mixed integer quadratically-constrained program. This formulation further involved constraints preventing virtually-intersecting frame elements, bounds on the cross-section areas, and a set of additional cut constraints to reduce the feasible design space. The emergent optimization problems assumed parametrization of the moments of inertia by degree-two polynomials of the cross-section areas, and were solved globally using a branch-and-bound method.

The remaining two methods are based on a fairly different concept: they provide hierarchies of relaxations of increasing size, hence avoid the need for enumeration. Using the sum-of-squares (SOS) hierarchy of specific semidefinite programming relaxations \citep{Kojima_2003}, \citet[Section 5.3]{Murota_2010} optimized the weight of frame structures while bounding the fundamental free-vibration eigenvalue from below. Despite theoretically guaranteed convergence of the objective function values \citep{Kojima_2006}, the associated optimal solutions may remain unknown.

Adopting a dual approach to \citet{Murota_2010}---the moment-sum-of-squares (MSOS) hierarchy \citep{Lasserre_2001,Henrion2006}---\citet{Tyburec2021} considered compliance optimization of volume-constrained frame and shell structures while accounting for multiple loading scenarios and self-weight. Similarly to the SOS hierarchy, the MSOS hierarchy guaranteed a convergence of the objective function values, but further maintained a simple procedure for extracting the global minimizers at the convergence~\citep{Henrion2006}. Furthermore, \citet{Tyburec2021} provided a simple method for projecting the relaxed solutions onto the feasible set of the original problem, thereby providing both lower- and feasible upper-bounds. These bounds naturally assess the quality of the relaxations and guarantee performance gap of the upper-bound designs with respect to the global minimizers. Finally, they showed that the optimality gap approaches zero if the global minimizer is unique.
 
\subsection{Aims and novelty}\label{sec:novelty}

In this contribution, we investigate global weight optimization of frame structures with bounded compliance of multiple loading scenarios. This is achieved by extending our original results for global compliance optimization of frame structures via the MSOS hierarchy \citep{Tyburec2021}.

In the MSOS hierarchy, exactness of relaxations follows from a rank condition. If not satisfied, the relaxation is not exact and its quality is generally uncertain. The goal of this paper is thus to estimate the quality of inexact relaxations for the weight optimization problem by providing a method for constructing feasible upper bounds from the relaxed solutions. These ingredients also settle a simple sufficient condition of global $\varepsilon$-optimality, vanishing in the limit for problems with the global minimizers forming a convex set.

In contrast to \citep{Toragay2022}, our work avoids the need for an enumerative exploration of the feasible space. Moreover, our setup naturally accommodates degree-three (and possibly higher-degree) polynomials, which are necessary, e.g., for the height optimization of rectangular cross-sections, or for thickness optimization of plates and shells \citep[Section 3.5]{Tyburec2021}.

This contribution is organized as follows. We start by a brief introduction to the moment-sum-of-squares hierarchy in Section \ref{sec:po} and formalizing the investigated optimization problem in Section \ref{sec:optproblem}. Section \ref{sec:complprop} provides basic mathematical properties of the compliance function, which we further exploit in Section \ref{sec:upperbound} to generate feasible upper bounds to the original optimization problem under mild assumptions. Finally, Section \ref{sec:msos} presents a modified optimization problem formulation that satisfies convergence assumptions of the moment-sum-of-squares hierarchy, constructs feasible upper bounds from the relaxations, and develops a certificate of global $\varepsilon$-optimality. We illustrate these theoretical findings numerically in Section \ref{sec:examples}, and summarize our contribution in Section \ref{sec:conclusion}.

\section{Moment-sum-of-squares hierarchy}\label{sec:po}

We start with a brief introduction to the moment-sum-of-squares hierarchy. We refer the reader to \citep{Lasserre_2001,Lasserre_2015,Henrion2006} for a more thorough treatment.

Let us consider optimization problems of the form
\begin{subequations}\label{eq:polyproblem}
\begin{align}
\min_{\mathbf{x}}\;& f(\mathbf{x})\\
\text{s.t.}\;& \mathbf{G}(\mathbf{x}) \succeq 0,\label{eq:Gcon}
\end{align}
\end{subequations}
where $f (\mathbf{x}):\mathbb{R}^n\mapsto \mathbb{R}$ and $\mathbf{G}(\mathbf{x}): \mathbb{R}^n \mapsto \mathbb{S}^{m}$ are real polynomial mappings and $\mathbb{S}^m$ stands for the space of $m\times m$ real symmetric square matrices. Further, the notation $\bullet \succ 0$ ($\bullet \succeq 0$) denotes positive definiteness (semi-definiteness) of $\bullet$ and $\mathcal{K}(\mathbf{G}(\mathbf{x}))$ represents the feasible set of \eqref{eq:Gcon}.

Let now $\mathbf{x}\mapsto\mathbf{b}_k (\mathbf{x})$ be the polynomial space basis of polynomials in $\mathbb{R}^n$ of degree at most $k$
\begin{equation}
\mathbf{b}_k (\mathbf{x}) = \left( 1\;\, x_1\;\, \dots\;\, x_n\;\, x_1^2\;\, x_1 x_2\;\, \dots\;\, x_n^2\;\, \dots\;\, x_n^k \right).
\end{equation}
Then, using a coefficient vector $\mathbf{q} \in \mathbb{R}^{\lvert \mathbf{b}_k(\mathbf{x}) \rvert}$, we can write any polynomial $p(\mathbf{x}): \mathbb{R}^n \mapsto \mathbb{R}$ of degree at most $k$ as a linear combination of the monomial entries in the basis $\mathbf{b}_k (\mathbf{x})$, i.e., $p(\mathbf{x}) = \mathbf{q}^\mathrm{T} \mathbf{b}_k(\mathbf{x})$.

Also, let $\left\{\bm{\alpha} \in \mathbb{N}^{n}: \mathbf{1}^\mathrm{T}\bm{\alpha} \le k, \prod_{i=1}^n x_i^{\alpha_i} \in \mathbf{b}_k(\mathbf{x}) \right\}$ be a multi-index and $\mathbf{y} \in \mathbb{R}^{\lvert \mathbf{b}_k(\mathbf{x}) \rvert}$ the moments of probability measures supported on $\mathcal{K}(\mathbf{G}(\mathbf{x}))$. In this work, we label the moment vector entries associated with the monomials $\prod_{i=1}^n x_i^{\alpha_i} \in \mathbf{b}_k (\mathbf{x})$ as $y_{\bm{\alpha}} = y_{\prod_{i=1}^{n} x_i^{\alpha_i}}$.

Furthermore, we need to set a formal definition of the (matrix) sum-of-squares decomposition:
\begin{definition}\label{def:sos}
	The matrix $\bm{\Sigma}(\mathbf{x}): \mathbb{R}^n\mapsto \mathbb{S}^{m}$ is a (matrix) sum-of-squares function if there exists a matrix $\mathbf{H}(\mathbf{x}): \mathbb{R}^n \mapsto \mathbb{R}^{m\times o}$ such that $\forall \mathbf{x}: \bm{\Sigma}(\mathbf{x}) = \mathbf{H}(\mathbf{x})\left[\mathbf{H}(\mathbf{x})\right]^\mathrm{T}$.
\end{definition}

Notice that for $m=1$, Definition \ref{def:sos} reduces to the case of scalar sum-of-squares polynomials. 

Using this definition and $\langle\mathbf{X}, \mathbf{Y} \rangle = \text{Tr}(\mathbf{X}\mathbf{Y}^\mathrm{T})$ to denote the inner product on $\mathbb{S}^m$, we provide the (Archimedean) assumption of algebraic compactness:

\begin{assumption}\label{ass:archimedean}\citep{Henrion2006} Assume that there exist sum-of-squares polynomials $p_0(\mathbf{x}): \mathbb{R}^n \mapsto \mathbb{S}^1$ and $\mathbf{R}(\mathbf{x}): \mathbb{R}^n \mapsto \mathbb{S}^m$ such that the superlevel set $\bigl\{ \mathbf{x}\in \mathbb{R}^n: p_0(\mathbf{x}) + \langle \mathbf{R}(\mathbf{x}), \mathbf{G}(\mathbf{x}) \rangle \ge 0 \bigr\}$ is compact.
\end{assumption}

If Assumption \ref{ass:archimedean} holds, \eqref{eq:polyproblem} is equivalent to an infinite-dimensional linear semidefinite (and hence convex) program
\begin{subequations}\label{eq:truncation}
\begin{align}
f^{(r)} = \min_\mathbf{y}\; & \mathbf{q}_0^\mathrm{T} \mathbf{y}\\
\mathrm{s.t.}\;& \mathbf{M}_{2r}(\mathbf{y}) \succeq 0\\
& \mathbf{M}_{2r-d}(\mathbf{G}\mathbf{y}) \succeq 0,
\end{align}
\end{subequations}
with the relaxation degree $r\rightarrow \infty$. For $r \in \mathbb{N}$ and finite, \eqref{eq:truncation} provides then a finite-dimensional truncation of \eqref{eq:polyproblem}. In \eqref{eq:truncation}, $d$ denotes the maximum degree of a polynomial in $\mathbf{G}(\mathbf{x})$, and $\mathbf{M}_{2r}(\mathbf{y})$ with $\mathbf{M}_{2r-d}(\mathbf{G}\mathbf{y})$ are the (truncated) moment and localizing matrices associated with the moments $\mathbf{y}$ and $\mathbf{G}\mathbf{y}$, respectively. We refer the reader to \citep[Section D]{Henrion2006} to more details.

With increased relaxation degree $r$, larger portions of the infinite-dimensional program are incorporated, so that a convergence to the optimum value $f^*$ of $f$ is obtained in the limit.

\begin{theorem}\label{th:convergence}\citep{Henrion2006}
	Let Assumption~\ref{ass:archimedean} be satisfied. Then, $f^{(r)}\nearrow f^{*}$ as $r \rightarrow \infty$.
\end{theorem}

However, the convergence is generically finite \citep{Nie_2013} and usually occurs at a~low $r$. In addition to the convergence of the objective function value, the global optimality can be recognized and the corresponding minimizers extracted using the flat extension theorem of \citet{Curto1996}. We again refer an interested reader to \citep{Henrion2006,Lasserre_2015} for more information.

\section{Methods}

\subsection{Optimization problem formulation}\label{sec:optproblem}

This paper deals with a global solution of the weight minimization \eqref{eq:original_weight} problem with bounded compliances of the $n_\mathrm{lc}$ load cases \eqref{eq:original_compliances} under linear-elastic equilibrium \eqref{eq:original_equilibrium} and non-negativity of the design variables \eqref{eq:original_areas}:
\begin{subequations}\label{eq:original}
\begin{align}
\min_{\mathbf{a},\mathbf{u}} \;& \sum_{e=1}^{n_\mathrm{e}} \rho_e \ell_e a_e\label{eq:original_weight}\\
\mathrm{s.t.}\;& \mathbf{K}_j(\mathbf{a}) \mathbf{u}_j = \mathbf{f}_j, \quad\forall j \in \{1,\dots, n_\mathrm{lc}\},\label{eq:original_equilibrium}\\
& \overline{c}_j - \mathbf{f}_j^\mathrm{T} \mathbf{u}_j \ge 0, \quad \forall j \in \{1,\dots,n_\mathrm{lc}\},\label{eq:original_compliances}\\
& \mathbf{a}\ge \mathbf{0}.\label{eq:original_areas}
\end{align}
\end{subequations}
In \eqref{eq:original}, $\mathbf{a}\in \mathbb{R}^{n_\mathrm{e}}_{\ge 0}$ is the vector of the design variables such as the cross-section areas of frames, $n_\mathrm{e}$ denotes the number of elements, $\bm{\ell} \in \mathbb{R}^{n_\mathrm{e}}_{>0}$ stands for a vector of volume multipliers so that the volume of the $e$-th element amounts to $\ell_e a_e$, and $\bm{\rho} \in \mathbb{R}^{n_\mathrm{e}}_{>0}$ are the element densities. Further, $\mathbf{\overline{c}} \in \mathbb{R}^{n_\mathrm{lc}}_{>0}$ are upper bounds for the compliance of the $n_\mathrm{lc}$ load cases, and $\mathbf{u}_j \in \mathbb{R}^{n_\mathrm{dof},j}$ with $\mathbf{f}_j \in \mathbb{R}^{n_\mathrm{dof},j}$ stand respectively for the generalized displacement and force vectors of the $j$-th load case, with $n_{\mathrm{dof},j}$ being the number of degrees of freedom. Without loss of generality, we assume that $\forall j \in \{1,\dots, n_\mathrm{lc}\}\}: \mathbf{f}_j \neq \mathbf{0}$.

For bending-resistant structures, such as frames and flat shells, the structural stiffness matrices $\mathbf{K}_j (\mathbf{a})$ follow from the assembly
\begin{equation}\label{eq:polydep}
\mathbf{K}_j(\mathbf{a}) = \mathbf{K}_{j,0} + \sum_{e=1}^{n_\mathrm{e}}\left[ \mathbf{K}^{(1)}_{j,e} a_{e} + \mathbf{K}^{(2)}_{j,e} a_{e}^2 + \mathbf{K}^{(3)}_{j,e} a_{e}^3\right],
\end{equation}
in which $\mathbf{K}_{j,0} \in \mathbb{S}_{\succeq 0}^{n_{\mathrm{dof},j}}$ constitutes a design-independent stiffness matrix, and $\forall i \in \{1,2,3\}: \mathbf{K}^{(i)}_{j,e} \in\mathbb{S}_{\succeq 0}^{n_{\mathrm{dof},j}}$ are portions of the $e$-th element stiffness matrix that depend on the monomials $a_{e}^i$ linearly.

For the optimization problem \eqref{eq:original}, it is natural to assume solvability of the equilibrium system \eqref{eq:original_equilibrium} and forbid rigid body motions if all optimized elements are present:
\begin{assumption}\label{ass:posdef}
	$\forall \mathbf{a}>\mathbf{0}, \forall j \in \{1,\dots,n_\mathrm{lc}\}: \mathbf{K}_j(\mathbf{a})) \succ 0.$
\end{assumption}
\noindent Then, we can reformulate \eqref{eq:original} equivalently to a non-linear semidefinite program, see, e.g., \citep{Achtziger2008,Kanno2011,Tyburec2021},
\begin{subequations}\label{eq:sdp}
\begin{align}
\min_{\mathbf{a}}\; & \sum_{e=1}^{n_\mathrm{e}} \rho_e \ell_e a_e\\
\text{s.t.}\;& \begin{pmatrix}
\overline{c}_j & -\mathbf{f}_j^\mathrm{T}\\
-\mathbf{f}_j & \mathbf{K}_j(\mathbf{a})
\end{pmatrix} \succeq 0, \quad \forall j \in \{1,\dots,n_\mathrm{lc}\},\label{eq:pmi}\\
&\mathbf{a} \ge \mathbf{0}.\label{eq:nonneg}
\end{align}
\end{subequations}
The problem \eqref{eq:sdp} is in general non-convex due to the polynomial nature of $\mathbf{K}_j(\mathbf{a})$, recall Eq.~\eqref{eq:polydep}. Nevertheless, \eqref{eq:pmi} maintains a special structure that is described next.

\subsection{Properties of the compliance function}\label{sec:complprop}

From \eqref{eq:pmi}, while utilizing \citep[Lemma 1]{Tyburec2021}, we can express the compliance function $c_j (\mathbf{a})$ of the $j$-th load case as
\begin{subequations}\label{eq:complfun}
	\begin{align}
	c_j(\mathbf{a}) &= \mathbf{f}_j^\mathrm{T} \mathbf{K}_j(\mathbf{a})^\dagger \mathbf{f}_j,\\
	\mathbf{f}_j &\in \mathrm{Im} \left(\mathbf{K}_j(\mathbf{a})\right),\label{eq:inimi}
	\end{align}
\end{subequations}
where $\bullet^\dagger$ denotes the Moore-Penrose pseudo-inverse of $\bullet$. In this section, we state basic properties of $c_j$.

\begin{proposition}\label{prop:quasiconvex}
Compliance $c_j$ is a~non-increasing function of $\mathbf{a}$.
\end{proposition}
\begin{proof}
Assume that $\{\mathbf{a}: \mathbf{f}_j \in \mathrm{Im}\left(\mathbf{K}_j(\mathbf{a})\right)\}$. Then, based on Appendix \ref{app:sensitivity}, the derivatives of $c_j$ at $\mathbf{a}$ read as
\begin{multline}
\frac{\partial c_j (\mathbf{a})}{\partial a_i} = -\sum_{e=1}^{n_\mathrm{e}}\left[\mathbf{f}_j^\mathrm{T}\mathbf{K}_j(\mathbf{a})^\dagger \frac{d \mathbf{K}_j(\mathbf{a})}{d a_e}\mathbf{K}_j(\mathbf{a})^\dagger \mathbf{f}_j\right].
\end{multline}
Since $\mathbf{K}^{(i)}_{j,e} \succeq 0$ and $\mathbf{a}\ge \mathbf{0}$ by definition, it holds that $\frac{d \mathbf{K}_j(\mathbf{a})}{d a_e} \succeq 0$. Thus, $\forall j \in \{1,\dots,n_\mathrm{e}\}:\frac{\partial c_j(\mathbf{a})}{\partial a_i} \le 0$ and $c_j(\mathbf{a})$ is a non-increasing function.
\end{proof}

Further, we investigate the range of $c_j$. Because $c_j (\mathbf{a})$ is continuous it suffices to find $\inf c_j(\mathbf{a})$ and $\sup c_j (\mathbf{a})$. To this goal and similarly to \citep[Appendix A]{Tyburec2019}, we partition $\mathbf{K}_j(\mathbf{a})$ and $\mathbf{f}_j$ in \eqref{eq:pmi} according to the dependence on the design variables $\mathbf{a}$ as follows. Let $\mathbf{U}_{\mathrm{N},j}$ be the orthonormal bases associated with $\mathrm{Ker}\left(\sum_{e=1}^{n_\mathrm{e}} \sum_{i=1}^3\mathbf{K}_{j,e}^{(i)} a_e^i\right)$ and let $\mathbf{U}_{\mathrm{R},j}$ be the bases of $\mathrm{Im}\left(\sum_{e=1}^{n_\mathrm{e}}\sum_{i=1}^3\mathbf{K}_{j,e}^{(i)} a_e^i\right)$. After projecting $\mathbf{K}_j(\mathbf{a})$ via these bases, we receive the partitioning
\begin{equation}
	\begin{pmatrix}
		\mathbf{U}_{\mathrm{R},j}^\mathrm{T}\\
		\mathbf{U}_{\mathrm{N},j}^\mathrm{T}
	\end{pmatrix} \mathbf{K}_j(\mathbf{a})
	\begin{pmatrix}
		\mathbf{U}_{\mathrm{R},j} & \mathbf{U}_{\mathrm{N},j}
	\end{pmatrix} = 
	\begin{pmatrix}
		\mathbf{K}_{\mathrm{A},j}(\mathbf{a}) & \mathbf{K}_{\mathrm{AB},j}^\mathrm{T}\\
		\mathbf{K}_{\mathrm{AB},j} & \mathbf{K}_{\mathrm{B},j}
	\end{pmatrix}
\end{equation}
in which, without loss of generality,
\begin{equation}\label{eq:posdef}
\forall \mathbf{a}>\mathbf{0}: \mathbf{U}_{\mathrm{R},j}^\mathrm{T} \left(\sum_{e=1}^{n_\mathrm{e}}\mathbf{K}_{j,e}^{(i)} a_e^i\right) \mathbf{U}_{\mathrm{R},j} \succ 0.
\end{equation}
Moreover, $\mathbf{K}_{\mathrm{A},j}(\mathbf{a})$ is the only part that depends on $\mathbf{a}$.

Similarly, we define 
\begin{equation}
\begin{pmatrix}
	\mathbf{f}_{\mathrm{A},j}\\ \mathbf{f}_{\mathrm{B},j}
\end{pmatrix} = \begin{pmatrix}
	\mathbf{U}_{\mathrm{R},j}^\mathrm{T}\\
	\mathbf{U}_{\mathrm{N},j}^\mathrm{T}
\end{pmatrix} \mathbf{f}_j.
\end{equation}
Then, \eqref{eq:pmi} can be equivalently rewritten\footnote{Notice, however, that the solution $\tilde{\mathbf{u}}_j$ to the transformed system $\begin{pmatrix}
		\mathbf{U}_{\mathrm{R},j}^\mathrm{T}& \mathbf{U}_{\mathrm{N},j}^\mathrm{T}
	\end{pmatrix}^\mathrm{T} \mathbf{K}_j(\mathbf{a})
	\begin{pmatrix}
		\mathbf{U}_{\mathrm{R},j} & \mathbf{U}_{\mathrm{N},j}
	\end{pmatrix} \tilde{\mathbf{u}}_j = \begin{pmatrix}
	\mathbf{U}_{\mathrm{R},j}^\mathrm{T}& \mathbf{U}_{\mathrm{N},j}^\mathrm{T}
\end{pmatrix}^\mathrm{T}\mathbf{f}_j$ differs from $\mathbf{u}_j$ in \eqref{eq:original_equilibrium}. The original vector field $\mathbf{u}_j$ can be recovered by another transformation as $\mathbf{u}_j = \begin{pmatrix}
\mathbf{U}_{\mathrm{R},j}& \mathbf{U}_{\mathrm{N},j}
\end{pmatrix}\tilde{\mathbf{u}}_j$.} as 
\begin{equation}\label{eq:constant}
\begin{pmatrix}
c_j & \mathbf{f}_{\mathrm{A},j}^\mathrm{T} & \mathbf{f}_{\mathrm{B},j}^\mathrm{T}\\
\mathbf{f}_{\mathrm{A},j} & \mathbf{K}_{\mathrm{A},j}(\mathbf{a}) & \mathbf{K}_{\mathrm{AB},j}^\mathrm{T}\\
\mathbf{f}_{\mathrm{B},j} & \mathbf{K}_{\mathrm{AB},j} & \mathbf{K}_{\mathrm{B},j}
\end{pmatrix} \succeq 0
\end{equation}
Because $\mathbf{K}_{\mathrm{B},j} \succ 0$ due to Assumption \ref{ass:posdef}, then, using the Schur complement lemma \citep{Haynsworth1968}, \eqref{eq:constant} is reducible to
\begin{align}\label{eq:schl}
	\begin{pmatrix}
	c_{\mathrm{sch},j} & -\mathbf{f}^\mathrm{T}_{\mathrm{sch},j}\\
	-\mathbf{f}_{\mathrm{sch},j} & \mathbf{K}_{\mathrm{sch},j} (\mathbf{a})
	\end{pmatrix}\succeq 0.
\end{align}
with
\begin{subequations}\label{eq:schur}
\begin{align}
\mathbf{f}_{\mathrm{sch},j} &= \mathbf{f}_{\mathrm{A},j} - \mathbf{K}_{\mathrm{AB},j}^\mathrm{T} \mathbf{K}_{\mathrm{B},j}^{-1} \mathbf{f}_{\mathrm{B},j}\\
\mathbf{K}_{\mathrm{sch},j} (\mathbf{a}) &= \mathbf{K}_{\mathrm{A},j}(\mathbf{a}) - \mathbf{K}_{\mathrm{AB},j}^\mathrm{T} \mathbf{K}_{\mathrm{B},j}^{-1} \mathbf{K}_{\mathrm{AB},j}\\
c_{\mathrm{sch},j}(\mathbf{a}) &= c_j(\mathbf{a}) - \mathbf{f}_{\mathrm{B},j}^\mathrm{T} \mathbf{K}_{\mathrm{B},j}^{-1} \mathbf{f}_{\mathrm{B},j}\label{eq:cschur}
\end{align}
\end{subequations}
being the condensed force vector, stiffness matrix, and compliance, respectively. Then, we are ready to prove the following proposition.
\begin{proposition}\label{prop:inf}
	For the partitioning in Eq.~\ref{eq:constant}, it holds that
	\begin{equation}
	\inf c_j (\mathbf{a}) = \mathbf{f}_{\mathrm{B},j}^\mathrm{T} \mathbf{K}_{\mathrm{B},j}^{-1} \mathbf{f}_{\mathrm{B},j}.
	\end{equation}
\begin{proof}
	Based on \eqref{eq:schl}, we have $c_{\mathrm{sch},j}(\mathbf{a}) \ge 0$. Hence, $c_j(\mathbf{a}) \ge \mathbf{f}_{\mathrm{B},j}^\mathrm{T} \mathbf{K}_{\mathrm{B},j}^{-1} \mathbf{f}_{\mathrm{B},j}$ due to \eqref{eq:cschur}. Finally, it suffices to show that $c_j(\mathbf{a}) \rightarrow \mathbf{f}_{\mathrm{B},j}^\mathrm{T} \mathbf{K}_{\mathrm{B},j}^{-1} \mathbf{f}_{\mathrm{B},j}$ for $\mathbf{a}\rightarrow \bm{\infty}$. Because $\forall i, e: \mathbf{K}_{j,e}^{(i)} \succeq 0$ and \eqref{eq:posdef}, the eigenvalues of $\mathbf{K}_{\mathrm{A},j}(\mathbf{a})$ approach infinity as $\mathbf{a}\rightarrow \bm{\infty}$. Hence,
	\begin{equation}
	c_{\mathrm{sch},j}(\mathbf{a})=\mathbf{f}^\mathrm{T}_{\mathrm{sch},j} \mathbf{K}_{\mathrm{sch},j}(\mathbf{a})^{-1} \mathbf{f}_{\mathrm{sch},j} \rightarrow 0 \text{ as } \mathbf{a}\rightarrow \bm{\infty}
	\end{equation} 
	and, therefore, $c_j(\mathbf{a}) \rightarrow \mathbf{f}_{\mathrm{B},j}^\mathrm{T} \mathbf{K}_{\mathrm{B},j}^{-1} \mathbf{f}_{\mathrm{B},j}$ based on \eqref{eq:cschur}.
\end{proof}
\end{proposition}

\begin{remark}\label{rem:inf0}
	For the case of $\mathbf{K}_{\mathrm{B},j} \in \mathbb{S}^{0}$, we have $\inf_{\mathbf{a}} c_j \rightarrow 0$.
\end{remark}

Next, we consider the supremum part.

\begin{proposition}\label{prop:sup}
	For the partitioning in Eq.~\ref{eq:constant}, it holds that
	\begin{enumerate}
		\item $\sup c_j(\mathbf{a}) = \mathbf{f}_{\mathrm{sch},j}^\mathrm{T} \left(\mathbf{K}_{\mathrm{A},j,0}- \mathbf{K}_{\mathrm{AB},j}^\mathrm{T} \mathbf{K}_{\mathrm{B},j}^{-1} \mathbf{K}_{\mathrm{AB},j}\right)^{-1} \mathbf{f}_{\mathrm{sch},j}$\\if  $\mathbf{f}_{\mathrm{sch},j} \in \mathrm{Im}(\mathbf{K}_{\mathrm{A},j,0}- \mathbf{K}_{\mathrm{AB},j}^\mathrm{T} \mathbf{K}_{\mathrm{B},j}^{-1} \mathbf{K}_{\mathrm{AB},j})$
		\item $\sup c_j(\mathbf{a})= \infty$ otherwise.
	\end{enumerate}
	\begin{proof}
		The first part follows from \eqref{eq:schl} and corresponds to the setting when fixed elements are able to transmit prescribed loading to supports. For the second part, setting $\mathbf{a}\rightarrow\mathbf{0}$ renders the displacement field arbitrarily large, and thus the compliance infinite.
	\end{proof}
\end{proposition}

\subsection{Upper bounds to program \eqref{eq:sdp} by scalarization}\label{sec:upperbound}

Using these compliance function properties, this section develops a method for obtaining feasible upper bounds to \eqref{eq:sdp} under mild assumptions.

Let $\tilde{\mathbf{a}} \in \mathbb{R}^{n_\mathrm{e}}_{\ge 0}$ be a vector of fixed ratios of the cross-section areas such that $\forall j \in \{1,\dots,n_\mathrm{lc}\}:\mathbf{f}_j \in \mathrm{Im}\left(\mathbf{K}_j(\tilde{\mathbf{a}})\right)$. Further, define a scaling parametrization of the cross-section areas via a parameter $\delta>0$, i.e., $\mathbf{a}(\delta) = \delta \tilde{\mathbf{a}}$. In what follows, we state the conditions under which the values in Propositions \ref{prop:inf} and \ref{prop:sup} remain valid even though we replace $c_j (\mathbf{a})$ with $c_j (\mathbf{a}(\delta))$.

\begin{proposition}\label{prop:infima}
	If $\mathbf{f}_{\mathrm{A},j} \in \mathrm{Im}\left(\mathbf{U}_{\mathrm{R},j}^\mathrm{T} \left(\sum_{e=1}^{n_\mathrm{e}}\mathbf{K}_{j,e}^{(i)} \tilde{a}_e^i\right) \mathbf{U}_{\mathrm{R},j}\right)$
	holds, then, $\inf_{\mathbf{a}} c_j = \inf_{\delta} c_j$.
\begin{proof}
	The proof follows from Proposition~\ref{prop:inf}.
\end{proof}
\end{proposition}

\begin{remark}
	The condition in Proposition \ref{prop:infima} may not be satisfied only if $\mathbf{K}_{j,0}\neq \mathbf{0}$ and $\exists e: \tilde{a}_e=0$. From the mechanical point of view, such situation corresponds to the case of carrying loads through elements with prescribed stiffness, although the optimized domain would allow load transfer through elements that are eliminated with $\tilde{a}_e=0$.
\end{remark}

\begin{remark}
	If $\mathbf{K}_{j,0} = \mathbf{0}$, then the condition in Proposition \ref{prop:infima} simplifies to $\mathbf{f}_j \in \mathrm{Im}(\mathbf{K}_j(\tilde{\mathbf{a}}))$.
\end{remark}

For the case of upper bounds in Proposition \ref{prop:sup}, the situation is considerably easier---they remain the same regardless of $\tilde{\mathbf{a}}$.

Finally, we can state the procedure for constructing feasible upper bounds.

\begin{proposition}\label{prop:upperbound}
	Let $\forall j \in \{1,\dots, n_\mathrm{lc}\}:\inf_{\delta} c_j (\delta)< \overline{c}_j$ hold and let the feasible set of \eqref{eq:original} have a non-empty interior. Then, there exists $\delta> 0$ such that $\delta \tilde{\mathbf{a}}$ is feasible to \eqref{eq:sdp}. Furthermore, $\delta$ follows from a solution to the univariate optimization problem
	\begin{subequations}\label{eq:uni}
	\begin{align}
	\min_{\delta}\; & \delta\\
	\mathrm{s.t.} \;& \mathbf{f}_j^\mathrm{T} \mathbf{K}_j(\delta)^{\dagger} \mathbf{f}_j \le \overline{c}_j, \quad \forall j \in \{1,\dots ,n_\mathrm{lc}\},\label{eq:compl_uni}\\
	& \delta > 0.
	\end{align}
	\end{subequations}
\begin{proof}
	Since $\sup_{\delta} c_j (\delta)$ does not depend on $\tilde{\mathbf{a}}$, Proposition \ref{prop:sup}, \eqref{eq:uni} is solvable whenever \eqref{eq:sdp} is due to the assumption on infimum. \eqref{eq:compl_uni} is a quasi-convex function due to Proposition~\ref{prop:quasiconvex} but also a continuous function due to the constant rank property of $\mathbf{K}_j(\delta)$ for all $\delta>0$ \citep{Penrose1955}, so that the equality sign in \eqref{eq:compl_uni} can always be satisfied for at least one load case.
\end{proof}
\end{proposition}

From the numerical perspective, the value of $\inf_{\delta} c_j$ is obtained via Proposition \ref{prop:inf} for a partitioning in \eqref{eq:constant} that follows from $\mathbf{a}(\delta)=\tilde{\mathbf{a}}\delta$.

Due to the quasi-convexity of \eqref{eq:compl_uni}, the optimal scaling factor $\delta$, and thus an (upper-bound) feasible solution to \eqref{eq:sdp}, can be found by a bisection-type algorithm.

\subsection{Moment-sum-of-squares hierarchy}\label{sec:msos}

In this section, we modify \eqref{eq:sdp} to be practically solvable to global optimality by the moment-sum-of-squares hierarchy, develop a sequence of feasible upper bounds, and settle a simple sufficiency condition of global $\varepsilon$-optimality in the spirit of \citep{Tyburec2021}.

\subsubsection{Polynomial programming reformulation}
For convergence guarantees of the moment-sum-of-squares hierarchy, we need to certify algebraic compactness of the feasible set, recall Assumption \ref{ass:archimedean} and Theorem \ref{th:convergence}. This can be secured by bounding the design variables through quadratic constraints \citep[Proposition 4]{Tyburec2021}.

To set these constraints, we first notice that while the lower bounds for $\mathbf{a}$ come directly from the problem formulation, recall \eqref{eq:nonneg}, the upper bounds can be established by exploiting the results in Section \ref{sec:upperbound}. In particular, for any fixed $\tilde{\mathbf{a}}>\mathbf{0}$, the condition in Proposition~\ref{prop:infima} is satisfied, allowing us to compute optimal scaling $\delta^*$ through the program \eqref{eq:uni}, and thus construct a~feasible upper-bound to \eqref{eq:sdp}. An upper-bound structural weight then amounts to
\begin{equation}\label{eq:ubweight}
\overline{w} = \delta^* \sum_{e=1}^{n_\mathrm{e}} \rho_e \ell_e \tilde{a}_e.
\end{equation}
Since \eqref{eq:ubweight} bounds the weight from above, none of the structural elements can exceed the weight $\overline{w}$ at the optimum. Therefore, the individual variables $a_e$ can be bounded as
\begin{equation}
	0 \le a_e \le \frac{\overline{w}}{\rho_e \ell_e}
\end{equation}
which is then equivalent to
\begin{equation}\label{eq:compactbound}
a_e \left( \frac{\overline{w}}{\rho_e \ell_e} - a_e\right)\ge 0.
\end{equation}

From the numerical perspective, it is further advantageous to scale the design variables, i.e., solve the optimization problem in terms of $\forall e \in \{1,\dots,n_{\mathrm{e}}\}:a_{\mathrm{s},e} \in [-1,1]$ rather than in $\forall e \in \{1,\dots,n_{\mathrm{e}}\}:a_e \in [0,\overline{w}/(\ell_e \rho_e)]$, which is achieved by inserting
\begin{equation}
a_e = \frac{a_{\mathrm{s},e}+1}{2} \frac{\overline{w}}{\rho_e \ell_e}.
\end{equation}
After these modifications, the final formulation reads as
\begin{subequations}\label{eq:finalformulation}
\begin{align}
\min_{\mathbf{a}_\mathrm{s}} \; & 0.5\overline{w}\left(n_\mathrm{e} + \mathbf{1}^\mathrm{T} \mathbf{a}_\mathrm{s}\right)\\
\mathrm{s.t.}\; & \begin{pmatrix}
\overline{c}_j & -\mathbf{f}_j^\mathrm{T}\\
-\mathbf{f}_j & \mathbf{K}_j(\mathbf{a}_\mathrm{s})
\end{pmatrix}\succeq 0, \quad\forall j \in \{1,\dots, n_\mathrm{lc}\},\\
& a_{\mathrm{s},e}^2 \le 1,\quad \forall j \in \{1,\dots, n_\mathrm{e}\}.
\end{align}
\end{subequations}

\subsubsection{Recovering feasible upper bounds and sufficient condition of global $\varepsilon$-optimality}

In order to solve \eqref{eq:finalformulation} globally, we generate a hierarchy of convex outer approximations of the feasible set $\mathcal{K}(\mathbf{G}(\mathbf{x}))$, recall Section \ref{sec:po}. The feasible set of these relaxations is described in terms of the moments $\mathbf{y}$ that are indexed in the polynomial space basis $\mathbf{b}_{2r}(\mathbf{a}_{\mathrm{sc}}) = \{1, a_{\mathrm{s},1},\dots, a_{\mathrm{s},n_\mathrm{e}},\dots \}$. Because the emerging relaxations are linear in $\mathbf{y}$, recall \eqref{eq:truncation}, we solve a sequence of convex linear semidefinite programming problems.

Let now $\mathbf{y}_{\mathbf{a}_\mathrm{s}^1}^{(r)}$ be the optimal first-order moments associated with degree-$1$ polynomials in $\mathbf{b}_{2r}(\mathbf{a_{\mathrm{sc}}})$ of the $r$-th degree relaxation. Unscaling these first-order moments provides us with an estimate on the optimal scaling factors $\tilde{\mathbf{a}}$, i.e., 
\begin{equation}\label{eq:atilde}
\tilde{a}_e = \frac{{y}_{a_{\mathrm{s},e}^1}^{(r)}+1}{2} \frac{\overline{w}}{\rho_e \ell_e}, \forall e \in \{1,\dots, n_\mathrm{e}\}.
\end{equation}
For $\tilde{\mathbf{a}}$, it holds that $\mathbf{f}_j \in \mathrm{Im}\left(\mathbf{K}_j(\tilde{\mathbf{a}})\right)$ by \cite[Proposition 6]{Tyburec2021}. Consequently, we show how to construct feasible upper bounds next.

\begin{theorem}\label{th:ubound}
	Let $\mathbf{y}_{\mathbf{a}_\mathrm{s}^1}^{(r)}$ be the optimal first-order moments associated with the $r$-th relaxation. If $\inf c_j (\delta)<\overline{c}_j$ holds for all $j \in \{1\dots n_\mathrm{lc}\}$ and the problem \eqref{eq:original} is solvable, then $\delta^* \tilde{\mathbf{a}}$, with $\tilde{\mathbf{a}}$ set as \eqref{eq:atilde} and $\delta^*$ computed based on Proposition \ref{prop:upperbound}, is a feasible (upper-bound) solution to \eqref{eq:original}.
\begin{proof}
	Same as in Proposition \ref{prop:upperbound}.
\end{proof}
\end{theorem}

\begin{remark}\label{rem:upp0}
	Due to $\mathbf{f}_j \in \mathrm{Im}\left(\mathbf{K}_j(\tilde{\mathbf{a}})\right)$, the assumptions of Theorem \ref{th:ubound} may lack to be satisfied only if $\mathbf{K}_{j,0}\neq \mathbf{0}$ and $\exists e: \tilde{a}_e=0$. For $\mathbf{K}_{j,0}=\mathbf{0}$, the upper bounds can always be constructed.
\end{remark}

Having the sequence of upper bounds in Theorem \ref{th:ubound}, and a natural sequence of lower bounds from the relaxations, we arrive at a simple condition of global $\varepsilon$-optimality.

\begin{proposition}\label{prop:sufficient}
	Let $\delta^* \tilde{\mathbf{a}}$ be a feasible (upper-bound) solution to \eqref{eq:original} constructed based on Theorem \ref{th:ubound}. Then,
	\begin{equation}
	\left(\delta^*-1\right)0.5 \overline{w}(n_\mathrm{e} + \mathbf{1}^\mathrm{T} \mathbf{y}_{\mathbf{a}_\mathrm{s}^1}^{(r)}) \le \varepsilon
	\end{equation}
	is a sufficient condition of global $\varepsilon$-optimality.
\end{proposition}

Finally, we also show that $\varepsilon\rightarrow 0$ for problems with minimizers forming a convex set.

\begin{theorem}
	Let $\delta^* \tilde{\mathbf{a}}$ be a feasible (upper-bound) solution to \eqref{eq:original} constructed based on Theorem \ref{th:ubound}. If the set of global minimizers is convex, then, as $r\rightarrow \infty$, 
	\begin{equation}
	\left(\delta^*-1\right)0.5 \overline{w}(n_\mathrm{e} + \mathbf{1}^\mathrm{T} \mathbf{y}_{\mathbf{a}_\mathrm{s}^1}^{(r)}) = 0.
	\end{equation}
\begin{proof}
	Because of Theorem \ref{th:convergence} and satisfied Assumption \ref{ass:archimedean}, for $r\rightarrow \infty$, optimization over a set $\mathcal{K}$ is equivalent to optimization over its convex hull $\text{Conv}(\mathcal{K})$ \citep[Proposition 7]{Tyburec2021}. By Assumption \ref{ass:archimedean}, $\mathcal{K}$ is compact, and, thus, $\text{Conv}(\mathcal{K})$ is too. Therefore, we can express the convex hull using its limits points $\mathbf{d}_1, \mathbf{d}_2,\dots$,
	\begin{equation}
	\text{Conv}\left(\mathcal{K}\right) = \text{Conv}\left( \cup_{i=1}^\infty \left\{ \mathbf{d}_i \right\} \right).
	\end{equation}
	Having assumed that the set of global minimizers is convex, there must exist a convex set $\text{Conv}\left( \cup_{i=1}^\infty \left\{ \mathbf{d}^*_i \right\} \right) \subseteq \text{Conv}\left(\mathcal{K}\right)$ with points $\mathbf{d}_i^*$ that are associated with the minimum. 
\end{proof}
\end{theorem}

\section{Examples}\label{sec:examples}

In this section, we demonstrate the capabilities of the presented method by means of two illustrations: a modular three-story structure, and a part design. All computations were performed on a personal laptop with $24$~GB of RAM and Intel$^\text{\textregistered}$ Core$^\text{\texttrademark}$ i5-8350U CPU. For optimization, we relied on the \textsc{Mosek}~\citep{Mosek} solver.

\subsection{24-element modular frame structure}
\begin{figure*}[!t]
	\begin{subfigure}[b]{0.32\linewidth}
		\begin{tikzpicture}
		\scaling{2.2}
		\point{a}{0.00}{0.0}; \notation{1}{a}{$\circled{a}$}[above left];
		\point{b}{1.50}{0.0}; \notation{1}{b}{$\circled{b}$}[above right];
		\point{c}{0.00}{1.0}; \notation{1}{c}{$\circled{c}$}[above left];
		\point{d}{1.50}{1.0}; \notation{1}{d}{$\circled{d}$}[above right];
		\point{e}{0.00}{2.0}; \notation{1}{e}{$\circled{e}$}[above left];
		\point{f}{1.50}{2.0}; \notation{1}{f}{$\circled{f}$}[above right];
		\point{g}{0.00}{3.0}; \notation{1}{g}{$\circled{g}$}[above left];
		\point{h}{1.50}{3.0}; \notation{1}{h}{$\circled{h}$}[above right];
		\point{i}{0.75}{1.0}; \notation{1}{i}{$\circled{i}$}[below=-0.5mm];
		\point{j}{0.75}{2.0}; \notation{1}{j}{$\circled{j}$}[below=-0.5mm];
		\point{k}{0.75}{3.0}; \notation{1}{k}{$\circled{k}$}[below=-0.5mm];
		\point{l}{0.75}{0.5}; \notation{1}{l}{$\circled{l}$}[above];
		\point{m}{0.75}{1.5}; \notation{1}{m}{$\circled{m}$}[above];
		\point{n}{0.75}{2.5}; \notation{1}{n}{$\circled{n}$}[above];
		\point{o}{0.75}{0.0};
		
		\beam{2}{a}{c}; \notation{4}{a}{c}[$1$];
		\beam{2}{b}{d}; \notation{4}{b}{d}[$2$];
		\beam{2}{c}{e}; \notation{4}{c}{e}[$3$];
		\beam{2}{d}{f}; \notation{4}{d}{f}[$4$];
		\beam{2}{e}{g}; \notation{4}{e}{g}[$5$];
		\beam{2}{f}{h}; \notation{4}{f}{h}[$6$];
		\beam{2}{c}{i}; \notation{4}{c}{i}[$7$];
		\beam{2}{i}{d}; \notation{4}{i}{d}[$8$];
		\beam{2}{e}{j}; \notation{4}{e}{j}[$9$];
		\beam{2}{j}{f}; \notation{4}{j}{f}[$10$];
		\beam{2}{g}{k}; \notation{4}{g}{k}[$11$];
		\beam{2}{k}{h}; \notation{4}{k}{h}[$12$];
		\beam{2}{a}{l}; \notation{4}{a}{l}[$13$];
		\beam{2}{l}{d}; \notation{4}{l}{d}[$14$];
		\beam{2}{b}{l}; \notation{4}{b}{l}[$15$];
		\beam{2}{l}{c}; \notation{4}{l}{c}[$16$];
		\beam{2}{c}{m}; \notation{4}{c}{m}[$17$];
		\beam{2}{m}{f}; \notation{4}{m}{f}[$18$];
		\beam{2}{d}{m}; \notation{4}{d}{m}[$19$];
		\beam{2}{m}{e}; \notation{4}{m}{e}[$20$];
		\beam{2}{e}{n}; \notation{4}{e}{n}[$21$];
		\beam{2}{n}{h}; \notation{4}{n}{h}[$22$];
		\beam{2}{f}{n}; \notation{4}{f}{n}[$23$];
		\beam{2}{n}{g}; \notation{4}{n}{g}[$24$];
		
		\support{3}{a};
		\support{3}{b};
		
		\load{1}{i}[90][0.7][0.0]; \notation{1}{i}{$1$}[above=3mm,xshift=1.5mm];
		\load{1}{j}[90][0.7][0.0]; \notation{1}{j}{$1$}[above=3mm,xshift=1.5mm];
		\load{1}{k}[90][0.7][0.0]; \notation{1}{k}{$1$}[above=3mm,xshift=1.5mm];
		\load{1}{c}[180][0.7][0.0]; \notation{1}{c}{$1$}[below=0mm, xshift=-3mm];
		\load{1}{e}[180][0.7][0.0]; \notation{1}{e}{$1$}[below=0mm, xshift=-3mm];
		\load{1}{g}[180][0.35][0.0]; \notation{1}{g}{$0.5$}[below=0mm, xshift=-3mm];
		
		\dimensioning{2}{o}{l}{4.25}[$0.5$];
		\dimensioning{2}{l}{i}{4.25}[$0.5$];
		\dimensioning{2}{i}{m}{4.25}[$0.5$];
		\dimensioning{2}{m}{j}{4.25}[$0.5$];
		\dimensioning{2}{j}{n}{4.25}[$0.5$];
		\dimensioning{2}{n}{k}{4.25}[$0.5$];
		\dimensioning{1}{a}{o}{-0.75}[$0.75$];
		\dimensioning{1}{o}{b}{-0.75}[$0.75$];
		\end{tikzpicture}
		\caption{}
	\end{subfigure}%
	\hfill\begin{subfigure}[b]{0.15\linewidth}
		\centering
		Elements \squared{1}--\squared{6}:\\
		\begin{tikzpicture}
		\scaling{1.0}
		\point{a}{0.0}{0.0};\point{a0}{1.0}{0.0};
		\point{b}{0.0}{0.1};\point{b0}{1.0}{0.1};
		\point{c}{0.45}{0.1};\point{c0}{0.55}{0.1};
		\point{d}{0.45}{0.9};\point{d0}{0.55}{0.9};
		\point{e}{0.0}{0.9};\point{e0}{1.0}{0.9};
		\point{f}{0.0}{1.0};\point{f0}{1.0}{1.0};
		\draw[black, thick, fill=black!25] (a0) -- (b0) -- (c0) -- (d0) -- (e0) -- (f0) -- (f) -- (e) -- (d) -- (c) -- (b) -- (a) -- cycle;
		\dimensioning{2}{a}{f}{1.25}[$10t_e$];
		\dimensioning{1}{a}{a0}{1.25}[$10t_e$];
		\draw [-stealth](0.25,0.60) -- (0.5,0.75);
		\draw [-stealth](0.25,0.60) -- (0.35,0.95);
		\draw [-stealth](0.25,0.60) -- (0.35,0.05);
		\draw (-0.2,0.6) -- node[above=-0.5mm]{$t_e$}(0.025,0.6) -- (0.25,0.60);
		\end{tikzpicture}\\
		\vspace{3mm}
		Elements \squared{13}--\squared{24}:\\
		\begin{tikzpicture}
		\point{a}{-1.0}{0};\point{a0}{0.0}{0.0};
		\draw[thick, fill=black!25] (0.0,0) arc (0:360:0.5);
		\draw[thick, fill=white] (-0.1,0) arc (0:360:0.4);
		\dimensioning{1}{a}{a0}{0.65}[$10t_e$];
		\draw [-stealth](-0.3,-0.1) -- (-0.05,0.05);
		\draw (-0.7,-0.1) -- node[above=-0.5mm]{$t_e$}(-0.4,-0.1) -- (-0.3,-0.1);
		\end{tikzpicture}\\
		\vspace{3mm}
		Elements \squared{7}--\squared{12}:\\
		\begin{tikzpicture}
		\scaling{1.0}
		\point{a}{0.0}{0.0};\point{a0}{1.0}{0.0};
		\point{b}{0.0}{0.1};\point{b0}{1.0}{0.1};
		\point{c}{0.45}{0.1};\point{c0}{0.55}{0.1};
		\point{d}{0.45}{1.9};\point{d0}{0.55}{1.9};
		\point{e}{0.0}{1.9};\point{e0}{1.0}{1.9};
		\point{f}{0.0}{2.0};\point{f0}{1.0}{2.0};
		\draw[black, thick, fill=black!25] (a0) -- (b0) -- (c0) -- (d0) -- (e0) -- (f0) -- (f) -- (e) -- (d) -- (c) -- (b) -- (a) -- cycle;
		\dimensioning{2}{a}{f}{1.25}[$20t_e$];
		\dimensioning{1}{a}{a0}{2.25}[$10t_e$];
		\draw [-stealth](0.25,0.60) -- (0.5,0.75);
		\draw [-stealth](0.25,0.60) -- (0.35,1.95);
		\draw [-stealth](0.25,0.60) -- (0.35,0.05);
		\draw (-0.2,0.6) -- node[above=-0.5mm]{$t_e$}(0.025,0.6) -- (0.25,0.60);
		\end{tikzpicture}
		\caption{}
	\end{subfigure}%
	\hfill\begin{subfigure}[b]{0.5\linewidth}
		\begin{tikzpicture}	
		\begin{axis}[xmin=0.9, xmax=3.1, ymin=0.035, ymax=0.24, axis lines=center, xlabel=$r$, ylabel=$w$, every axis y label/.style={at=(current axis.above origin),anchor=south}, every axis x label/.style={at=(current axis.right of origin),anchor=west},yticklabel style={/pgf/number format/fixed, /pgf/number format/precision=2},xtick={1,...,3},width=\textwidth,height=9.2cm,legend pos=south east,]
		\addplot[color=black,mark=x] coordinates {
			(1, 0.046850)
			(2, 0.101032)
			(3, 0.114697)
		};
		\addlegendentry{lower bounds};
		\addplot+[color=black,dashed,mark=x] coordinates {
			(1, 0.147282)
			(2, 0.123123)
			(3, 0.115027)
		};
		\addlegendentry{upper bounds};
		\node at (rel axis cs:0.135,0.77) {\includegraphics[height=3.3cm]{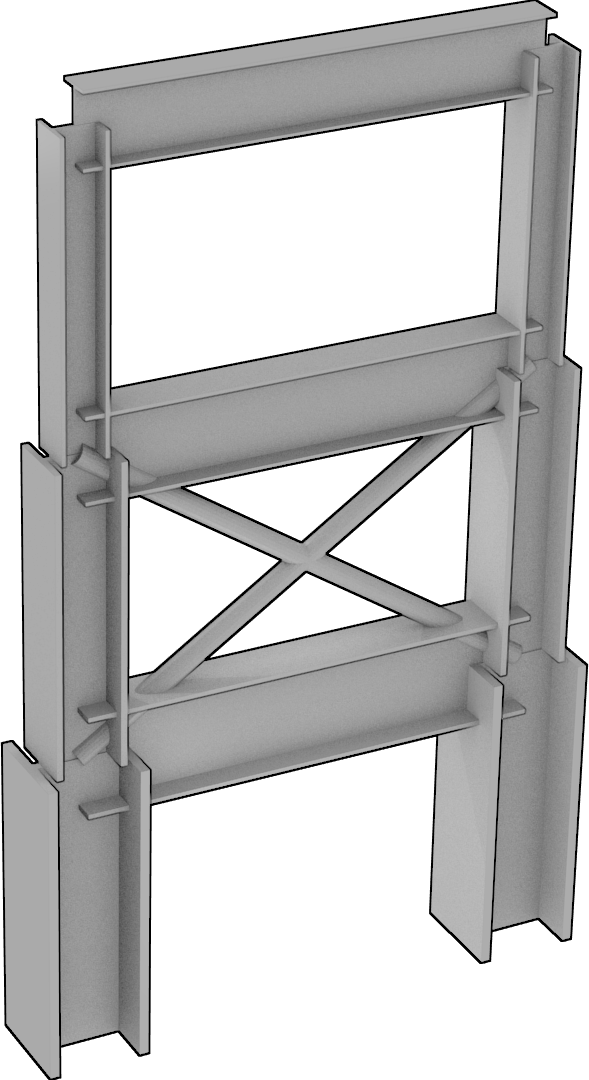}};
		\node at (rel axis cs:0.5,0.7) {\includegraphics[height=3.3cm]{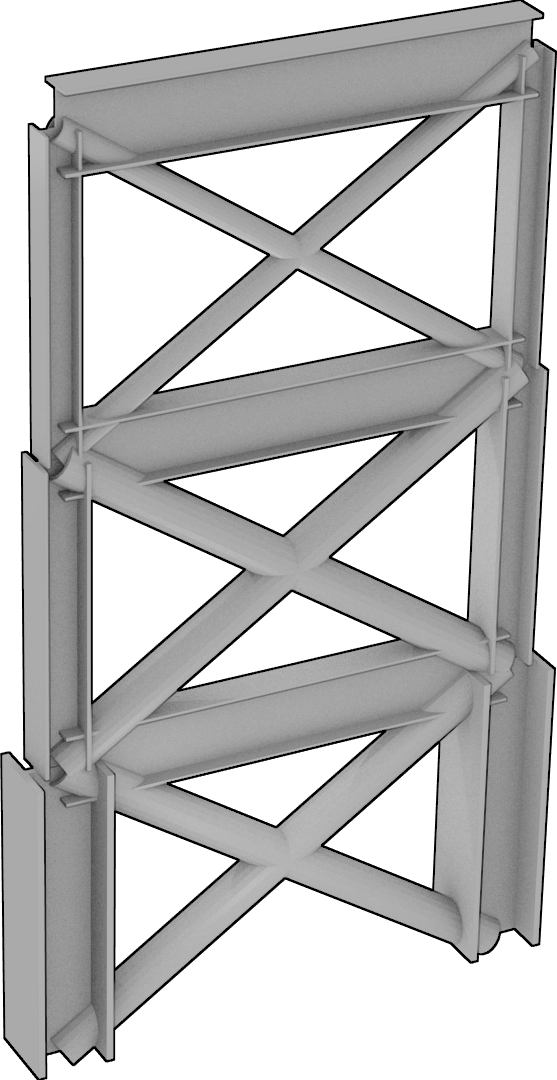}};
		\node at (rel axis cs:0.88,0.65) {\includegraphics[height=3.3cm]{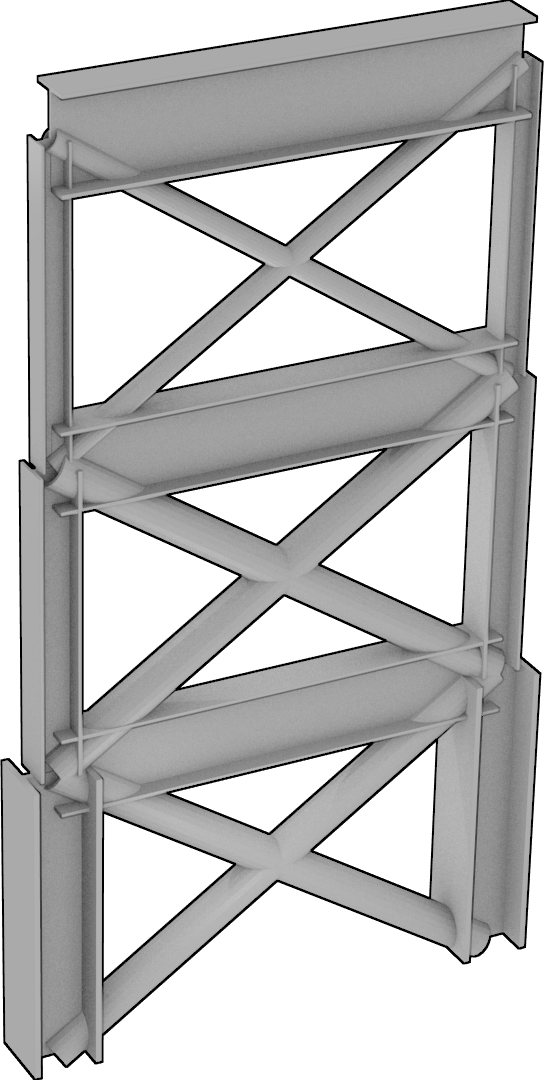}};
		\end{axis}
		\end{tikzpicture}
		\caption{}
	\end{subfigure}
	\caption{$24$-element frame structure: (a) boundary conditions, (b) cross-section parametrization, and (c) convergence of the proposed relaxation-based approach with visualized feasible upper-bound designs.}
	\label{fig:frame24}
\end{figure*}
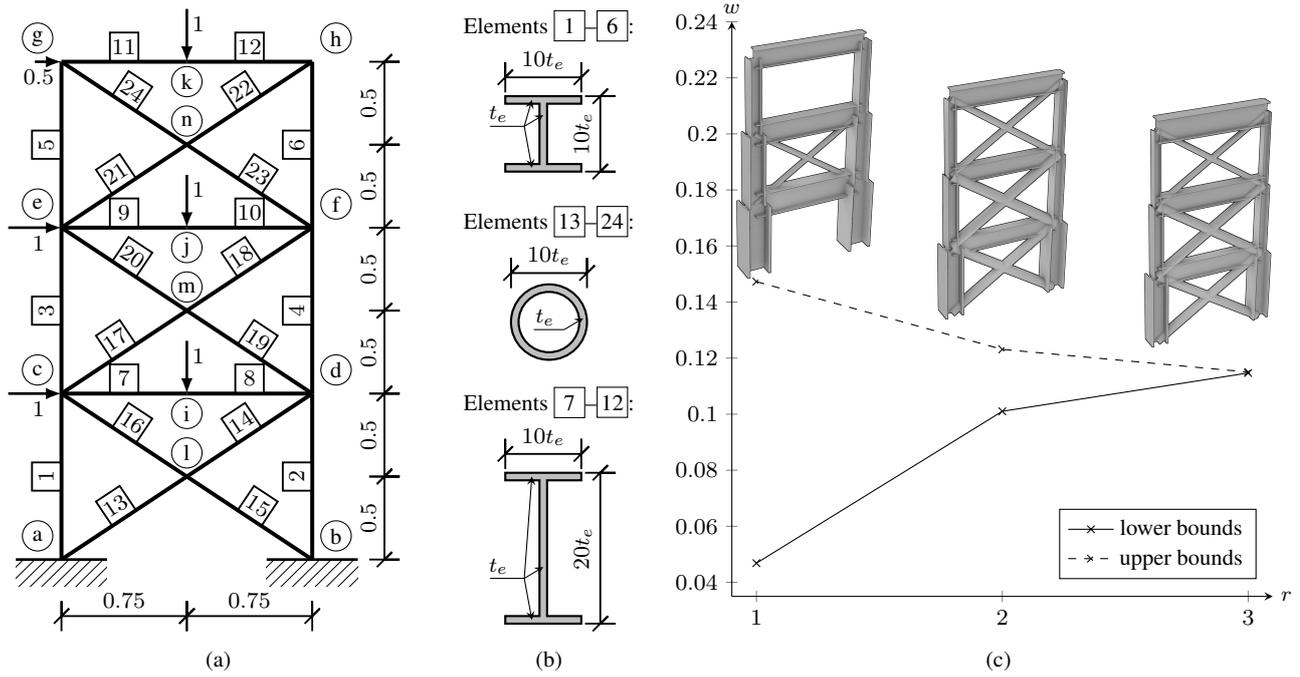

As the first illustration we investigate a modular frame structure containing $24$ Euler-Bernoulli finite elements and $36$ degrees of freedom, see Fig.~\ref{fig:frame24}a. For simplicity, we assume the dimensionless Young modulus $E=1.0$, density $\rho=1.0$, as well as the dimensionless structural dimensions.

The frame structure is clamped at the bottom nodes $\circled{a}$ and $\circled{b}$, and subjected to horizontal loads at nodes $\circled{c}$, $\circled{e}$ and $\circled{g}$, and to vertical forces acting at $\circled{i}$, $\circled{j}$ and $\circled{k}$.

We split the structural elements into several groups: first, we use a different cross-section parametrization within the structural elements, Fig.~\ref{fig:frame24}b. In addition, we also specify groups of elements that must have the same cross-section size to maintain structural symmetry: while for the columns, we set $a_1=a_2$, $a_3=a_4$, $a_5=a_6$, we require $a_7=a_8$, $a_9=a_{10}$ and $a_{11}=a_{12}$ for the horizontal beams. Finally, we enforce equal cross-section sizes within the circular tubes in a single story, i.e., $a_{13}=a_{14}=a_{15}=a_{16}$, $a_{17}=a_{18}=a_{19}=a_{20}$, and $a_{21}=a_{22}=a_{23}=a_{24}$. Hence, we have nine independent cross-section areas in total.

\begin{table}[!b]
	\centering
	\begin{tabular}{lrrrrr}
		$r$ & LB & UB & Time [s] & $n_\mathrm{c}\times m$ & $n$\\
		\hline
		$1$ & $0.047$ & $0.147$ & $0.05$ & $10$, $9\times 1$, $37$ & $54$ \\
		$2$ & $0.101$ & $0.123$ & $14.24$ & $55$, $9\times 10$, $370$ & $714$ \\
		$3$ & $0.115$ & $0.115$ & $13\thinspace473.50$ & $220$, $9\times 55$, $2\thinspace035$ & $5\thinspace004$
	\end{tabular}
	\caption{$24$-element frame structure optimization. LB abbreviates lower bound, UB stands for feasible upper bounds, and $r$ is the relaxation number. Further, $n_\mathrm{c}\times m$ denotes the number of $n_\mathrm{c}$ semidefinite constraints of the size $m$, and $n$ is the number of variables.}
	\label{tab:frame24}
\end{table}

Because all structural elements are being optimized, the term $\mathbf{K}_{0}$ in \eqref{eq:polydep} is empty. Based on Remark \ref{rem:inf0}, we thus have $\inf_{\mathbf{a}} c = 0$. As any positive compliance can thus be attained, we set the compliance upper-bound to $\overline{c}=5\thinspace000$. Starting with a uniform distribution of cross-section areas, $\tilde{\mathbf{a}}=\mathbf{1}$, the optimization problem \eqref{eq:uni} yields the optimal scaling factor $\delta=6.64 \times 10^{-3}$, which provides an upper-bound weight of $\overline{w} = 0.141$. 

Moreover, for any feasible first-order moments, Remark \ref{rem:upp0} assures us that feasible upper bounds can always be constructed. In the lowest, first-degree relaxation, we receive the lower-bound weight of $0.047$. Using the cross-section area distribution provided by the optimal first-order moments to construct a feasible upper bound, recall Theorem \ref{th:ubound}, we receive the weight $0.147$. In the second relaxation, we obtain the lower bound objective $0.101$, from which we recover an upper-bound weight $0.123$. Final, third relaxation yields a lower-bound weight of $0.115$, and the projected upper-bound design of weight $0.115$, see Table~\ref{tab:frame24}. Hence, the global optimality of the design is certified based on Proposition \ref{prop:sufficient}. Similarly, the hierarchy also converged based on the flat extension theorem of \citet{Curto1996}, allowing for extracting the unique global minimizer visualized in Fig.~\ref{fig:frame24}c.

\subsection{Part design}

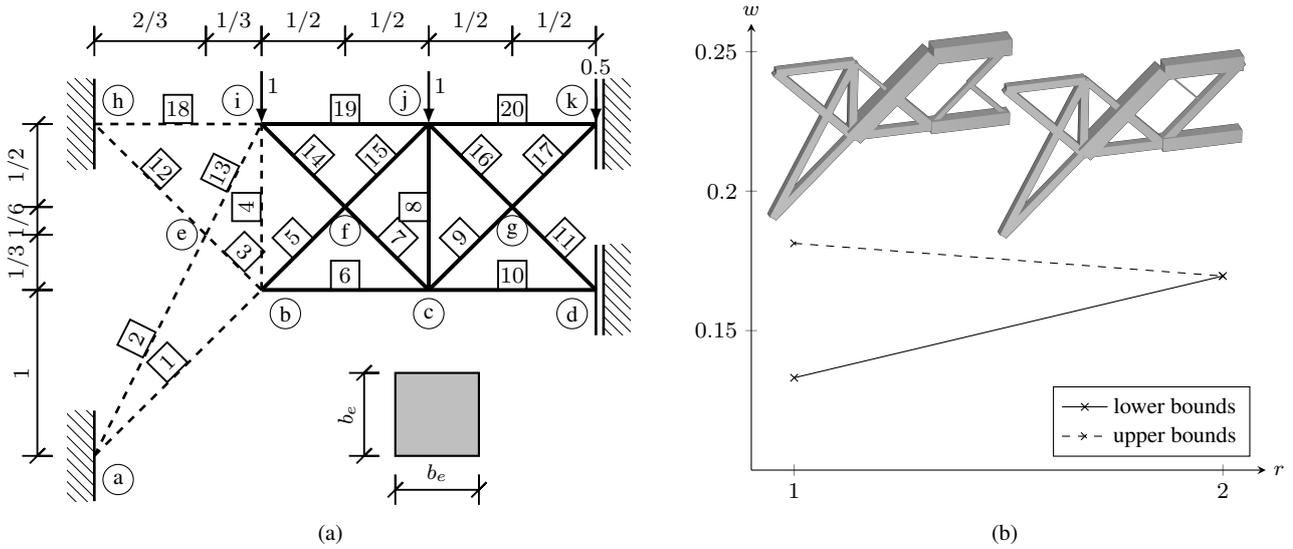
\begin{figure*}[!t]
	\begin{subfigure}[b]{0.5\linewidth}
		\begin{tikzpicture}
			\scaling{2.2}
			\point{a}{0.0}{0.0}; \notation{1}{a}{$\circled{a}$}[below right];
			\point{b}{1.0}{1.0}; \notation{1}{b}{$\circled{b}$}[below right];
			\point{c}{2.0}{1.0}; \notation{1}{c}{$\circled{c}$}[below];
			\point{d}{3.0}{1.0}; \notation{1}{d}{$\circled{d}$}[below left];
			\point{e}{0.667}{1.333}; \notation{1}{e}{$\circled{e}$}[left];
			\point{f}{1.5}{1.5}; \notation{1}{f}{$\circled{f}$}[below];
			\point{g}{2.5}{1.5}; \notation{1}{g}{$\circled{g}$}[below];
			\point{h}{0.0}{2.0}; \notation{1}{h}{$\circled{h}$}[above right];
			\point{i}{1.0}{2.0}; \notation{1}{i}{$\circled{i}$}[above left];
			\point{j}{2.0}{2.0}; \notation{1}{j}{$\circled{j}$}[above left];
			\point{k}{3.0}{2.0}; \notation{1}{k}{$\circled{k}$}[above left];
			
			\beam{3}{a}{b}; \notation{4}{a}{b}[$1$];
			\beam{3}{a}{e}; \notation{4}{a}{e}[$2$];
			\beam{3}{b}{e}; \notation{4}{b}{e}[$3$];
			\beam{3}{b}{i}; \notation{4}{b}{i}[$4$];
			\beam{2}{b}{f}; \notation{4}{b}{f}[$5$];
			\beam{2}{b}{c}; \notation{4}{b}{c}[$6$];
			\beam{2}{c}{f}; \notation{4}{c}{f}[$7$];
			\beam{2}{c}{j}; \notation{4}{c}{j}[$8$];
			\beam{2}{c}{g}; \notation{4}{c}{g}[$9$];
			\beam{2}{c}{d}; \notation{4}{c}{d}[$10$];
			\beam{2}{d}{g}; \notation{4}{d}{g}[$11$];
			\beam{3}{e}{h}; \notation{4}{e}{h}[$12$];
			\beam{3}{e}{i}; \notation{4}{e}{i}[$13$];
			\beam{2}{f}{i}; \notation{4}{f}{i}[$14$];
			\beam{2}{f}{j}; \notation{4}{f}{j}[$15$];
			\beam{2}{g}{j}; \notation{4}{g}{j}[$16$];
			\beam{2}{g}{k}; \notation{4}{g}{k}[$17$];
			\beam{3}{h}{i}; \notation{4}{h}{i}[$18$];
			\beam{2}{i}{j}; \notation{4}{i}{j}[$19$];
			\beam{2}{j}{k}; \notation{4}{j}{k}[$20$];
			
			\support{3}{a}[-90];
			\support{3}{h}[-90];
			\support{4}{d}[90];
			\support{4}{k}[90];
			
			\load{1}{i}[90][0.7][0.0]; \notation{1}{i}{$1$}[above=3mm,xshift=1.5mm];
			\load{1}{j}[90][0.7][0.0]; \notation{1}{j}{$1$}[above=3mm,xshift=1.5mm];
			\load{1}{k}[90][0.35][0.0]; \notation{1}{k}{$0.5$}[above=5.5mm,xshift=0mm];
			
			\dimensioning{1}{h}{e}{5.5}[$2/3$];
			\dimensioning{1}{e}{i}{5.5}[$1/3$];
			\dimensioning{1}{i}{f}{5.5}[$1/2$];
			\dimensioning{1}{f}{j}{5.5}[$1/2$];
			\dimensioning{1}{j}{g}{5.5}[$1/2$];
			\dimensioning{1}{g}{k}{5.5}[$1/2$];
			\dimensioning{2}{a}{b}{-0.75}[$1$];
			\dimensioning{2}{b}{e}{-0.75}[$1/3$];
			\dimensioning{2}{e}{f}{-0.75}[$1/6$];
			\dimensioning{2}{f}{h}{-0.75}[$1/2$];
			
			\point{aa}{1.8}{0.0}; \point{bb}{2.3}{0.5};
			\point{cc}{1.8}{0.5}; \point{dd}{2.3}{0.0};
			\draw[black, thick, fill=black!25] (aa) -- (cc) -- (bb) -- (dd) -- cycle;
			\dimensioning{1}{aa}{dd}{-0.45}[$b_e$];
			\dimensioning{2}{dd}{bb}{3.55}[$b_e$];
		\end{tikzpicture}
		\caption{}
	\end{subfigure}%
	\hfill\begin{subfigure}[b]{0.48\linewidth}
		\begin{tikzpicture}	
			\begin{axis}[xmin=0.9, xmax=2.1, ymin=0.100, ymax=0.26, axis lines=center, xlabel=$r$, ylabel=$w$, every axis y label/.style={at=(current axis.above origin),anchor=south}, every axis x label/.style={at=(current axis.right of origin),anchor=west},yticklabel style={/pgf/number format/fixed, /pgf/number format/precision=2},xtick={1,2},width=\textwidth,height=7.5cm,legend pos=south east,]
				\addplot[color=black,mark=x] coordinates {
					(1, 0.133097)
					(2, 0.169563)
				};
				\addlegendentry{lower bounds};
				\addplot+[color=black,dashed,mark=x] coordinates {
					(1, 0.181292)
					(2, 0.169563)
				};
				\addlegendentry{upper bounds};
				\node at (rel axis cs:0.27,0.77) {\includegraphics[height=2.5cm]{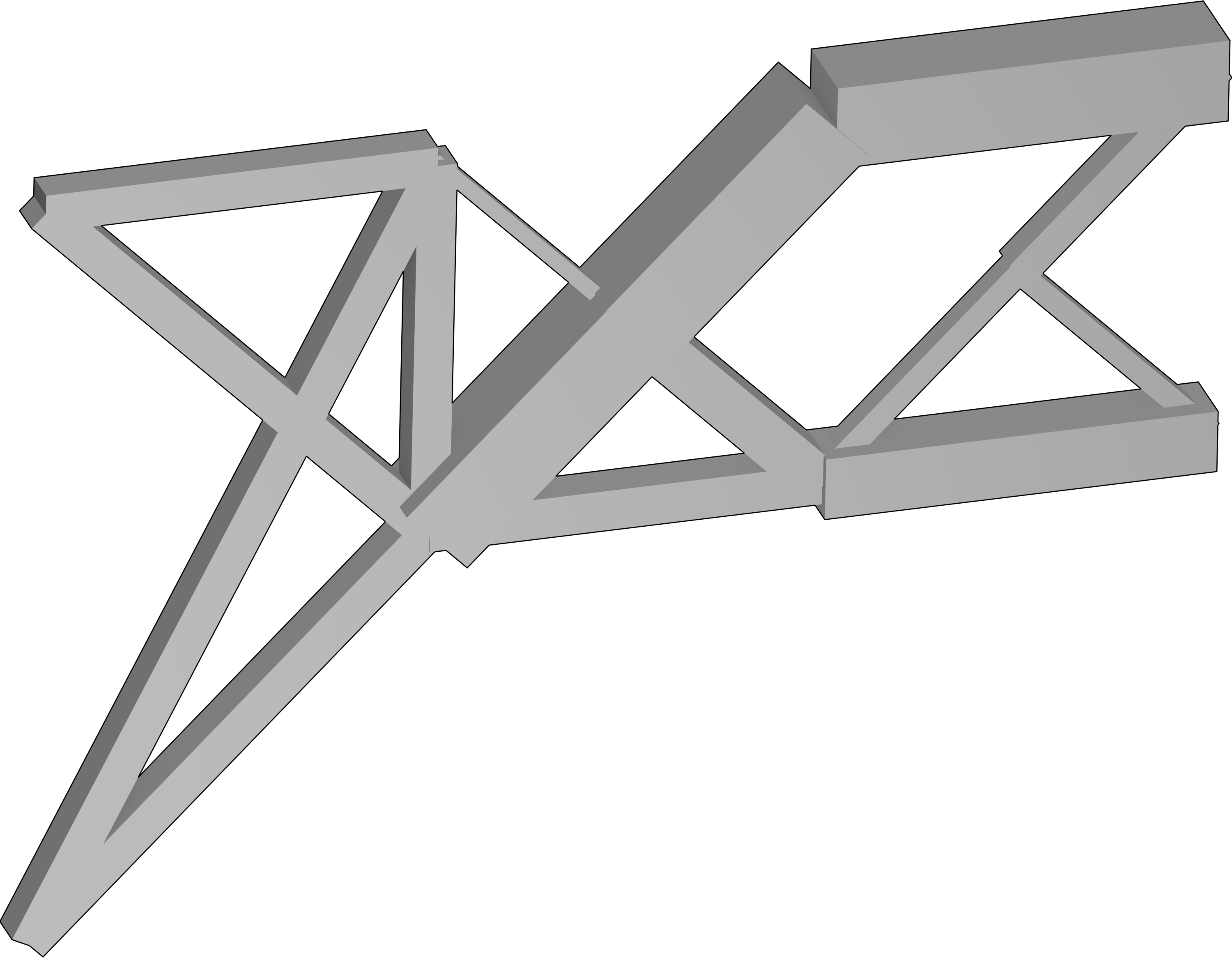}};
				\node at (rel axis cs:0.72,0.73) {\includegraphics[height=2.5cm]{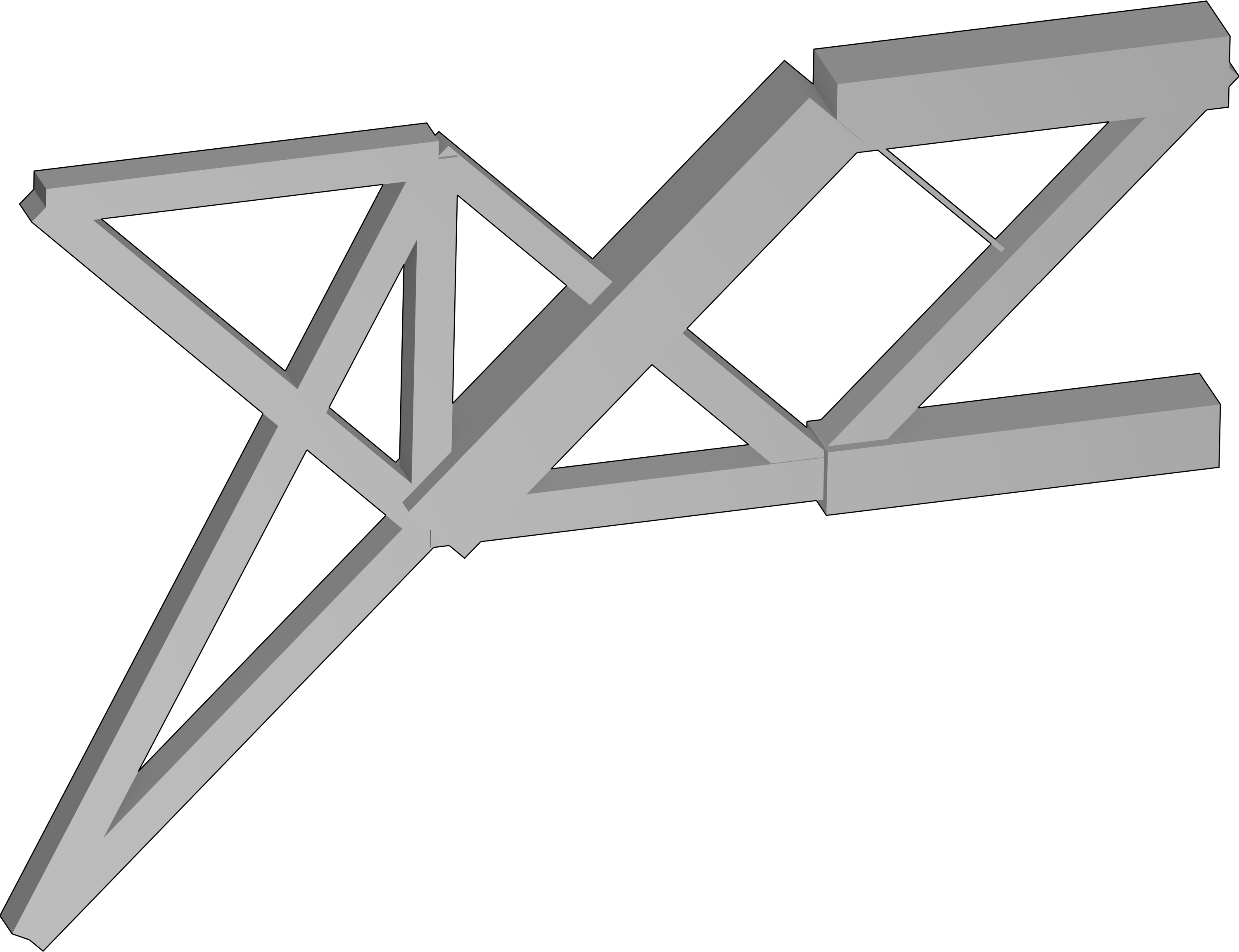}};
			\end{axis}
		\end{tikzpicture}
		\caption{}
	\end{subfigure}
	\caption{Frame structure reinforcement problem. (a) Discretization, boundary conditions and cross-section parametrization, and (b) convergence of the proposed relaxation-based approach with visualized feasible upper-bound designs.}
	\label{fig:frame20}
\end{figure*}

Second, we consider the problem of optimizing a structural part. Such problems appear, e.g., in stiffening and reinforcing structure design or in structural component optimizations.

Here, we assume the problem shown in Fig.~\ref{fig:frame20}a, consisting of $12$ nodes that are interconnected with $20$ Euler-Bernoulli beam elements. All these elements have square cross-sections. While the elements drawn in Fig~\ref{fig:frame20}a with a solid line are subject of optimization, the elements denoted with dashed lines share the cross-section area of $0.01$. Further, we have again the dimensionless Young modulus $E=1.0$ and the density $\rho=1.0$.

The structure is fully clamped at the nodes $\circled{a}$ and $\circled{h}$, whereas symmetric boundary conditions are assumed along the $\circled{d}$--$\circled{k}$ axis. The structure is loaded by unitary vertical forces at the nodes $\circled{i}$ and $\circled{j}$, and by a halved vertical force at the node $\circled{k}$.

Because of the fixed elements $\squared{1}$--$\squared{4}$, $\squared{12}$, $\squared{13}$, and $\squared{18}$, the term $\mathbf{K}_0$ is now present in \eqref{eq:polydep}. Using Proposition \ref{prop:inf}, we thus receive a positive value of the compliance infimum, $\inf_{\mathbf{a}} c = 578.9$, which can be approached by the optimized elements only in the limit. Thus, we set $\overline{c} = 1\thinspace000$.

\begin{table}[!b]
	\centering
	\begin{tabular}{lrrrrr}
		$r$ & LB & UB & Time [s] & $n_\mathrm{c}\times m$ & $n$\\
		\hline
		$1$ & $0.133$ & $0.181$ & $0.04$ & $14$, $13\times 1$, $19$ & $104$ \\
		$2$ & $0.170$ & $0.170$ & $29.52$ & $105$, $13\times14$, $266$ & $2\thinspace379$
	\end{tabular}
	\caption{Part design optimization. LB abbreviates lower bound, UB stands for feasible upper bounds, and $r$ is the relaxation number. Further, $n_\mathrm{c}\times m$ denotes the number of $n_\mathrm{c}$ semidefinite constraints of the size $m$, and $n$ is the number of variables.}
	\label{tab:frame20}
\end{table}

After setting $\tilde{\mathbf{a}}=\mathbf{1}$, the optimization problem \eqref{eq:uni} provides us with the optimal scaling factor $\delta = 2.678 \times 10^{-2}$ that is associated with the upper-bound solution of the weight $\overline{w} = 0.285$. Using this bound to make the design space compact and solving the emerging optimization problem \eqref{eq:finalformulation} via the MSOS hierarchy, we receive the lower bound weight of $0.133$ in the first relaxation and the associated feasible upper-bound design of the weight $0.181$. The second relaxation makes the hierarchy converge with respect to both the optimality gap and the flatness of the moment matrices ranks \citep{Curto1996}. The optimal design, shown in Fig.~\ref{fig:frame20}, weights $0.170$. In all the relaxations, setting $\tilde{\mathbf{a}}$ based on the first-order moments produced the same value of $\inf_{\delta} c = 578.9$.

The hierarchy convergence appear summarized in Table~\ref{tab:frame20}. We note here that because of the partitioning in \eqref{eq:constant} and five constant rows/columns, we adopted the Schur complement lemma to reduce the problem size and accelerate its solution.

\section{Results and discussion}\label{sec:conclusion}

In this contribution, we have extended our previous results for global compliance optimization of bending resistant structures \citep{Tyburec2021} to the weight minimization setting. To this goal, we have first exploited monotonicity of the compliance function, and developed a univariate problem \eqref{eq:uni} for computing feasible upper bounds to the weight optimization problem \eqref{eq:original} when the ratio of cross-section areas is fixed. We proposed to solve this problem by a bisection-type algorithm.

Based on such constructed upper bound, it is possible to bound the design variables from above, and thus show that the assumption of algebraic compactness, which is needed for the convergence of the Lasserre hierarchy, is satisfied. Developing and solving an efficient polynomial programming formulation, we have shown that, under mild assumptions, the first order moments from the relaxations may serve as the ratios of the cross-section areas, enabling a~construction of feasible upper bounds in each relaxation. Finally, a comparison of the relaxations lower bounds with the constructed upper bounds establishes a simple sufficient condition of global $\varepsilon$-optimality, and this condition converges to zero in the limit in the case of a~convex set of global minimizers.

We have illustrated these theoretical results on a set of two optimization problems. These problems revealed applicability of the approach to small-scale problems and a rapid convergence of the hierarchy.

We plan to extend our approach in several directions. First, we are interested in problems in structural dynamics such as eigenvalue problems \citep{Achtziger2008} and steady-state harmonic oscillations. Second, we aim to investigate methods for accelerating the optimization process, e.g., by exploiting structural \citep{Zheng2021,Kocvara2020} and term sparsity, or adopting the very recent results in optimization problems with tame structure \citep{Aravanis2022}.

\paragraph{Funding} Marek Tyburec and Michal Ko\v{c}vara acknowledge the support of the Czech Science foundation through project No. 22-15524S. Martin Kru\v{z}\'ik appreciated the support of the Czech Science Foundation via project No. 21-06569K, and by the Ministry of Education, Youth and Sports through the mobility project 8J20FR019. We also acknowledge support by European Union's Horizon 2020 research and innovation program under the Marie Sk\l odowska-Curie Actions, grant agreement 813211 (POEMA).

\section*{Declarations}

\paragraph{Conflict of interest} The authors declare no competing interests.

\paragraph{Replication of results} Source codes are available at \citep{tyburec_sw}.

\appendix
\section{Sensitivity of the compliance function}\label{app:sensitivity}

Here, we derive gradient of the compliance function
\begin{equation}\label{eq:compl}
c_j = \mathbf{f}_j^\mathrm{T} \left[\mathbf{K}_j(\mathbf{a})\right]^{\dagger} \mathbf{f}_j,
\end{equation}
where $\mathbf{f}_j \in \text{Im}(\mathbf{K}_j (\mathbf{a}))$.

We start by writing the pseudo-inverse identity:
\begin{equation}
\left[\mathbf{K}_j(\mathbf{a})\right]^\dagger = \left[\mathbf{K}_j(\mathbf{a})\right]^\dagger \mathbf{K}_j(\mathbf{a}) \left[\mathbf{K}_j(\mathbf{a})\right]^\dagger
\end{equation}
Using the chain rule while exploiting symmetry of $\mathbf{K}_j(\mathbf{a})$ and $\left[\mathbf{K}_j(\mathbf{a})\right]^\dagger$, we obtain
\begin{multline}\label{eq:chaineq}
\left(
2 \left[\mathbf{K}_j(\mathbf{a})\right]^\dagger \mathbf{K}_j(\mathbf{a}) - \mathbf{I}
\right)
\frac{\partial \left[\mathbf{K}_j(a_e)\right]^\dagger}{\partial a_e} =\\
-\left[\mathbf{K}_j(\mathbf{a})\right]^\dagger \frac{\partial \mathbf{K}_j(\mathbf{a}) }{\partial a_e} \left[\mathbf{K}_j(\mathbf{a})\right]^\dagger.
\end{multline}
Next, we need to show that
\begin{equation}\label{eq:diffimage2}
\frac{\partial \left[\mathbf{K}_j(a_e)\right]^\dagger}{\partial a_e} \in \text{Im}\left(\left[\mathbf{K}(\mathbf{a})\right]^\dagger\right),
\end{equation}
or, equivalently,
\begin{equation}\label{eq:diffimage}
\frac{\partial \mathbf{K}_j(a_e)}{\partial a_e} \in \text{Im}\left(\mathbf{K}(\mathbf{a})\right).
\end{equation}
Indeed, since $\mathbf{a} \ge \mathbf{0}$ and $\forall e,i: \mathbf{K}_{j,e}^{(i)} \succeq 0$, we observe that
\begin{subequations}
\begin{align}
\text{Span}\left(\mathbf{K}_j(\mathbf{a})\right) = \cup_{e=1}^{n_\mathrm{e}}\cup_{i=1}^3 \text{Span}\left(\mathbf{K}^{(i)}_{j,e}\right) \cup \text{Span}\left(\mathbf{K}_{j,0}\right)\\
\text{Span}\left(\frac{\partial\mathbf{K}_j(a_e)}{\partial a_e}\right) = \cup_{i=1}^3 \text{Span}\left(\mathbf{K}^{(i)}_{j,e}\right),
\end{align}
\end{subequations}
so that $\text{Span}\left(\frac{\partial\mathbf{K}_j(a_e)}{\partial a_e}\right) \subseteq \text{Span}\left(\mathbf{K}_j(a_e)\right)$. Consequently, the conditions \eqref{eq:diffimage2} with \eqref{eq:diffimage} hold true.

Therefore, since $\left[\mathbf{K}_j(\mathbf{a})\right]^\dagger \mathbf{K}_j(\mathbf{a})$ is a projector onto the range of $\left[\mathbf{K}_j(\mathbf{a})\right]^\dagger$, the gradient of \eqref{eq:compl} evaluates as
\begin{equation}
\frac{\partial c_j(\mathbf{a})}{\partial a_e} =
-\mathbf{f}_j^\mathrm{T}\left[\mathbf{K}_j(\mathbf{a})\right]^\dagger \frac{\partial \mathbf{K}_j(\mathbf{a}) }{\partial a_e} \left[\mathbf{K}_j(\mathbf{a})\right]^\dagger\mathbf{f}_j.
\end{equation}

\bibliography{liter.bib}
\bibliographystyle{springernat}
\end{document}